\documentclass{article}

\usepackage{amssymb,amscd,amsfonts,amsbsy}
\usepackage{enumerate}
\usepackage{amsmath,amsthm,amssymb,amsfonts}
\usepackage{mathrsfs}
\usepackage{epsf,epsfig}
\usepackage{amsmath}
\usepackage{amsfonts}
\usepackage{amssymb}
\usepackage{comment}
\usepackage{xcolor}
\providecommand{\U}[1]{\protect\rule{.1in}{.1in}}
\usepackage{bm} 

\newtheorem{theorem}{Theorem}

\newtheorem{corollary}[theorem]{Corollary}

\newtheorem{lemma}[theorem]{Lemma}

\newtheorem{proposition}[theorem]{Proposition}
\newtheorem{remark}[theorem]{Remark}

\def\R{\mathbb{R}}

\def\dest{\Delta^*}

\def\dist{\mathcal{D}'(\S)}

\def\RR{\mathbb{R}^3}
\def\R{\mathbb{R}}
\def\B{\mathbb{B}}
\def\S{\mathbb{S}}

\def\yl0{\mathcal{Y}_\ell^0}
\def\gradesf{\nabla_\sigma}
\def\vectL2{L^2(\S,\R^3)}

\def\ML2{L^2_{+}(\S,\R^3)}

\def\mL2{L^2_{-}(\S,\R^3)}

\def\0L2{L^2_{0}(\S,\R^3)}
\def\lpm{\bm{L}_-^p(\mathbb{S})}

\def\max{\mathcal{N}}
\def\esfarm{Y_{\ell,m}}
\title{{Hardy spaces for the Lam\'e equation}
\thanks{
The first and fourth authors were supported by Spanish grant
PID2021-
124195NB-C31, the second and the third by the Mexican grant PAPIIT-UNAM IN104224.}
\thanks{2000 {\it Math Subject Classification.} Primary:  42B30, 46E15, ; Secondary: 74B05. 
\newline
{\it Key words}: Hardy spaces, Elasticity theory, Lam\'e equation, Elastic Poisson kernel, Riesz systems, zonal multiplier. }}
\author{ Barcel\'o J.A., Per\'ez-Esteva S., Marmolejo-Olea E., Vilela M. C.}

\begin{document}
\maketitle


\begin{abstract}
We study, for $1 \leq p \leq \infty$, the Hardy space $\bm{h}_e^p(\B)$, the elastic analogue of the classical Hardy spaces of harmonic functions in the unit ball of $\mathbb{R}^3$. The space consists of vector-field solutions of the Lamé system satisfying the standard integrability condition on concentric spheres centered at the origin. Using the elastic Poisson kernel, we establish a Fatou-type theorem and show that $\bm{h}_e^p(\B)$ is isomorphic to the $\mathbb{R}^3$-valued Lebesgue space $L^p$ on the unit sphere for $1 < p \leq \infty$, while $\bm{h}_e^1(\B)$ corresponds to the space of $\mathbb{R}^3$-valued Borel measures on the unit sphere.  

For $1 < p < \infty$, we prove that $\bm{h}_e^p(\B)$ decomposes as the direct sum of  three subspaces. The main contribution of this paper is to describe each of these subspaces along with the corresponding spaces of boundary values. In particular, two of these spaces consist of solutions of the Lamé equation for all eligible choices of the  Lamé constants: one of them is the space of Riesz fields (solutions of the generalized Cauchy--Riemann equations) in $\bm{h}_e^p(\B)$; the second is the  space of fields given by the cross product of $x$ with such Riesz fields. The results rely on the classical decomposition of $L^2$ vector fields on the sphere into the direct sum of three spaces of vector spherical harmonics, which we extend to $L^p$. 
\end{abstract}
\section{Introduction}

The purpose of this paper is to study the structure of Hardy spaces of solutions of the Lam\'e equation  
\begin{equation*}
   \Delta^*\bm{u}=0
\end{equation*}
in the unit ball $\B$ in $\R^3$,
where 
\begin{equation*}
\Delta^*=\mu\Delta +(\lambda+\mu)\nabla div,
\end{equation*}
and the Lam\'e coefficients $\lambda$ and $\mu$
satisfy $ \mu>0$ and $2\mu +\lambda>0$, making $\Delta^*$ an elliptic operator acting on vector fields of $\R^3$ (see for example \cite{MitreaD}. These spaces, which we will denote by $\bm{h}^p_e(\B)$, are endowed with the classical norm from the Hardy space theory, consisting of the supremum of the  $\bm{L}^p$ norms on concentric spheres centered at the origin and radius less than one. Since for  $\lambda=-\mu$ the solution of the Lamé equations are identical to the vector-valued harmonic functions, then these spaces   generalize the Hardy spaces of harmonic functions (see \cite{Colzani T, Colzani, Axler}).  As expected,  the space of boundary limits of  elements in $\bm{h}^p_e(\B)$ is the $\R^3$-valued Lebesgue  $\bm{L}^p$ space in the sphere for $1<p\leq\infty$, and for  $\bm{h}^1_e(\B)$, it corresponds to  the space of borel measures in the unit sphere with values in $\R^3$. 

For the Lamé equation and other applications it is natural  to use of certain vector spherical harmonics (see  \cite{Freeden}) that provide a decomposition of the vector $\bm{L}^2(\mathbb{S})$ space into three orthogonal subspaces. We will prove that for $1<p<\infty$, a decomposition of the $\bm{L}^p$ space in the sphere holds as well, leading  via the elastic Poisson transform to a decomposition of $\bm{h}^p_e(\B)$ as a Banach direct product of three subspaces. The aim of this paper  is the precise description of each of these spaces. This will provide insight of the local structure of the Lamé equation. 

Remarkably, one of them will consist of Riesz fields,  namely fields $\bm{u}\in \bm{h}^p_e(\B) $ that satisfy the generalized Cauchy-Riemann equations $div\,\bm{u}=0$ and $\nabla\times \bm{u}=0$. It is well known the relevance of these fields in harmonic analysis. These vector fields are common elements of   $\bm{h}^p_e(\B)$ for every  admissible values of $\lambda$ and $\mu$. Moreover in this case, the vector field $x\times \bm{u}\in\bm{h}^p_e(\B)$  for such $\lambda$ and $\mu$.

The paper is organized as follows: In Sections 2 and 3 we will be dedicated to give definitions and to state all the preliminary results needed in the paper like Sobolev space and distributions in the sphere.  In Section 4 we will study the basic structure of the elastic Hardy spaces including the elastic Poisson transform and we will prove Fatou theorems for these Hardy spaces. In Section 5 we will review the mentioned vector-valued spherical harmonics decomposition  and we will prove the decompositions of the vector $\bm{L}^p$ spaces of the sphere. In Section 6 we complete the aim of this paper: to characterize the elastic extensions of vector fields on each of the $\bm{L}^p$ subspaces defined in Section 5. This will be achieved after describing in detail the three boundary $\bm{L}^p$ subspaces.   

\section{Hardy spaces for the Lam\'e operator}

In this section we  define the elastic Hardy spaces $\bm{h}^p_e(\B)$.  If $\bm{u}:\B\rightarrow\RR$, we say that $\bm{u} \in \bm{h}^p_e(\B)$, with $1\leq p<\infty$, if it is a solution of the Lamé system 
\begin{equation}
    \label{ecuacion}
    \dest \bm{u}=0,
\end{equation}
and satisfies
\begin{equation*}
\Vert \bm{u}\Vert_{\bm{h}^p_e}=\sup_{0\leq r<1}\left(\int_\S\vert \bm{u}(rw)\vert^p d\sigma(w)\right)^{1/p} < \infty ,
\end{equation*}
where $\mathbb{S}$ is the unit sphere in $\mathbb{R}^3$.
We say that $\bm{u} \in \bm{h}^\infty_e(\B)$ if it is a bounded solution of \eqref{ecuacion}.

The maximal elastic Hardy spaces 
are defined as follows:
for $\xi\in\S$ we define 

\begin{equation*}
 \Gamma_\xi=\B\cap conv \left(\lbrace\xi\rbrace\cup \lbrace x \in \mathbb{R}^3: \;  \vert x\vert\leq 1/2\rbrace\right),
 \end{equation*}
 where for $A$ a subset of $\R^3$, $conv(A)$ is the convex hull of $A$.

 We say that a solution $\bm{u}:\B\rightarrow\R ^3 $ of \eqref{ecuacion} belongs to $\bm{H}_e^ p(\mathbb{B}),\,  1\leq p \leq \infty,$ if 
\begin{equation}
\mathcal{N} \bm{u}(\xi)= \sup_{x \in \Gamma_\xi} | \bm{u}(x)|   
\end{equation}
belongs to $L^p(\S)$. Obviously $\bm{H}_e^p(\B)\subset \bm{h}_e^p(\B)$. 

The harmonic Hardy spaces in $\B$ are defined for scalar harmonic functions and are denoted by $h^p(\B)$ and $H^p(\B)$ same norms as before. 

\section{Preliminaries}

The following notation will be used throughout the paper: for non-negative quantities
X and Y we will use $X\lesssim Y$ ($X\gtrsim Y$) to
denote the existence of a positive constant $c$ such that $X\leq c\,Y$ ($X\geq c\,Y$). We will write
$X\sim Y$ if both $X\lesssim Y$ and $X\gtrsim Y$.

The following obvious estimate will be used throughout the paper.

\begin{equation}\label{integral}
 (1-r)\int_0^1 \frac{1}{(1-rt)^2}dt\leq C, 
\end{equation}
for some $C>0$ and all $r\in [0,1).$

Given a function $f$ on $\B$ we denote for every $r\in [0,1)$ the function $f_r$ defined on 
$\S$ by $f_r(\eta)=f(r\eta), \, \eta\in\S.$

We denote by  $ \bm{M}(\S) $ the space of $\R^3$-valued Borel measures in $\S$. Every $\bm{\mu}\in \bm{M}(\S)$ can be represented by $\bm{\mu}=(\mu_1,\mu_2,\mu_3),$
 where  each $\mu_i$ is a signed Borel measure in $\S$.  We define 
 \begin{equation*}
 \vert \bm{\mu}\vert=\sum_1^3 \vert \mu_i\vert
 \end{equation*}
 where  $\vert \mu_i\vert$ is the variation of the signed measure $\mu_i$.  Actually,
  $\vert \bm{\mu}\vert$ is the variation of the vector measure  $\bm{\mu}$  (see\cite{Diestel}) with respect to the norm $\vert \bm{x}\vert_1=\vert x_1\vert +\vert x_2\vert +\vert x_3\vert$ in $\R^3$.
  
Let $\lbrace Y_{\ell,m}\rbrace_{\ell=0,1, \cdots,  \; |m|\leq \ell},$
be an orthonormal basis  of real spherical harmonics for $L^2(\S)$ and let $\mathcal{H}^\ell$  denote the space of all spherical harmonics of degree $\ell$, which is generated by   $\{Y_{\ell,m}\}_{ |m| \leq \ell}$.
For any function $f$ on $L^1(\S)$ we denote $$\hat{f}_{\ell,m}=\int_\S f(\eta)Y_{\ell,m}(\eta)d\sigma(\eta).$$
We denote  by $\Delta_\sigma$ and  $\nabla_{\sigma}$, the Laplacian  (the Laplace-Beltrami operator)  and the gradient on $\S$. 
For a function $f$  defined on $\B$ we write $\Delta_{\sigma} f(x)$, where $x=|x|\frac{x}{|x|}=rx'$, for $\Delta_{\sigma} f(r\cdot)(x') $ and the same for the spherical gradient $\nabla_{\sigma}f(rx')=\nabla_{\sigma}f(r\cdot)(x').$  

 Recall that $\mathcal{H}^\ell$ is the eigenspace of $-\Delta_\sigma$ corresponding the  eigenvalue $\ell(\ell+1)$.
We have the estimates in $\ell$, (see for example \cite[page 225]{Neri}), 

\begin{equation}\label{estima ylm}
\Vert\esfarm \Vert_\infty=O(\ell^{1/2}) \text{ and } \Vert\nabla_\sigma\esfarm \Vert_\infty=O(\ell^{3/2}).
\end{equation}

For every $Y\in\mathcal{H}^\ell$ we denote its harmonic extension to  $\B$ by $Y(x)=r^\ell Y(x')$. If $f$ is a real function we define the vector function in $\S$,  $$\bm{f}^\vee(\xi) =f(\xi)\xi.$$
Then 
\begin{equation}\label{grad arm esf}
\nabla Y(x)= r^{\ell-1}\left(\ell \bm{Y}^\vee(x')+\nabla_{\sigma}Y(x')\right),
\end{equation}
where $\nabla$ denotes the gradient in $\R^3$.

For any sequence $\{\beta_\ell\}_{\ell \geq 0}$  in $\R$ we  define the zonal multiplier
 
\begin{equation}\label{multiplicadores}
\mathcal{M}_{\beta_\ell} (f)=\sum_{\ell \geq 0} \sum_{|m| \leq \ell} \beta_\ell a_{\ell,m}Y_{\ell,m}, 
\end{equation}
for $f=\sum_{\ell \geq 0} \sum_{|m| \leq \ell}a_{\ell,m}Y_{\ell,m}$ in the linear span of $\cup_{\ell 
 \geq 0}\mathcal{H}^\ell$.

Denote by  $\bm{C}^\infty(\S)$  the space of $C^\infty$-vector fields defined in the sphere, and by $\dist$  the space of distributions in $\S$, namely the space of continuous $\R$-linear functionals $f: C^\infty(\S)\rightarrow\R$. Likewise, denote  by $\bm{D}'(\mathbb{S})$,  the space of continuous $\R^3-$valued linear functions  on  $\bm{C}^\infty(\S)$.

Every  $f\in\dist$ can be written as
 $$f=\sum_{\ell \geq 0} \sum_{|m| \leq \ell} a_{\ell, m}Y_{\ell, m} $$ where  $ a_{\ell, m}=f(Y_{\ell, m})$  and the convergence is in the weak$^*$ topology. 
 The same holds for $\bm{f}\in \bm{D}'(\mathbb{S})$. In either case, there exist $N>0$  such that,  $ \vert a_{\ell, m}\vert\lesssim \ell ^N$, uniformly for $|m| \leq \ell,$ and this characterizes   $\dist $ and  $\bm{D}'(\mathbb{S})$.
 
 The action of such $f$ on $\varphi \in C^\infty(\S)$ is given by
 \begin{equation*}
 f(\varphi)=a_{0,0}\int_\S \varphi d\sigma+\int_{\S}\left( \sum_{\ell > 0} \sum_{|m| \leq \ell}   \frac{a_{\ell, m}}{(\ell(\ell+1))^M}Y_{\ell, m}\right)(-\Delta_\sigma)^{M}\varphi d\sigma
 \end{equation*}
 \begin{equation*}
 =a_{0,0}\int_\S \varphi d\sigma+\sum_{\ell > 0} \sum_{|m| \leq \ell}   \frac{a_{\ell, m}}{(\ell(\ell+1))^M}\int_{\S}Y_{\ell, m}(-\Delta_\sigma)^{M}\varphi d\sigma.
 \end{equation*}
 for $M$ large enough to  assure the convergence the series.

If the sequence $\{\beta_\ell\}_{\ell\geq 0}$ has polynomial growth  then $\mathcal{M}_{\beta_\ell}:\dist\rightarrow\dist.$ 
Let $E\subset\dist$ be a topological linear space  containing  all the zonal spaces  $\mathcal{H}^\ell$.
We say that $\mathcal{M}_{\beta_\ell}$ is a zonal multiplier for $E$  if $\mathcal{M}_{\beta_\ell}$ has a  continuous extension to $E$. In particular, for $M>0$, $(-\Delta_\sigma)^{M}$ is a zonal multiplier  for $\dist$ with the weak$^*$ topology, associated with the sequence $\left\{\ell^M\left( \ell +1 \right)^M\right\}_{\ell \geq 0}$.

We have the following result by R. S. Strichartz \cite{Str}.
 \begin{theorem}\label{Strichartz}
 Let $1<p<\infty$. If $h:(0, \infty) \rightarrow \mathbb{R}$ satisfies
\begin{equation}\label{Strichartj}
\left| x^k h^{(k)} (x) \right| \leq A, \hspace{0.5cm} x \in (0, \infty), \;\; k=0,1, \cdot \cdot \cdot,N,
\end{equation}
then $\mathcal{M}_{ h_\ell}$   is bounded in $L^p(\mathbb{S})$ for $\left| \frac{1}{2}-\frac{1}{p}  \right| < \frac{N}{2}$, where $h_\ell = h(\ell), \ell \in \mathbb{N}$.
\end{theorem}

The following result, (see \cite{Bakry}), it will be used throughout the paper.

\begin{theorem}\label{theorem_bakry} (D. Bakry) 
 Let  $1 <p < \infty$ and $f\in C^1(\S)$. Then
\begin{equation}  \label{bakry}
  \| \left( -\Delta_\sigma  \right)^{1/2} f \|_{L^p} \sim  \| \nabla_\sigma f \|_{\bm{L}^p} .  
\end{equation}

\end{theorem}

Define the Sobolev space $W^{1,p}(\S)$ as the space of the functions   $f\in L^p(\mathbb{S})$  such that $(-\Delta_\sigma)^{1/2} f\in L^p(\S).$

\begin{lemma} \label{lema_3}
Let $1< p<\infty$ and  $f\in  L^p(\S),$ then $f\in W^{1,p}(\S)$ if and only if there exists a sequence $\varphi_n\in C^\infty(\S)$ such that $\varphi_n\rightarrow f$ in $L^p(\S)$ and $\nabla_{\sigma}\varphi_n\rightarrow \bm{F}$ in $\bm{L}^p(\S)$ for some vector field $\bm{F}\in \bm{L}^p(\S)$.
In this case we denote $\nabla_\sigma f= \bm{F}$ (which is independent of the sequence $\varphi_n$).

Also
\begin{equation}\label{Bakry en Sobolev}
    \Vert  (-\Delta_\sigma)^{1/2} f \Vert_{L^p} \sim \Vert \nabla_\sigma f \Vert_{\bm{L}^p}   .
\end{equation}
for  $f\in W^{1,p}(\S).$
\end{lemma}
\begin{proof}
Let $f\in W^{1,p}(\S)$.  There exists a sequence $\varphi_n\in C^\infty(\S)$ such that $(-\Delta_\sigma)^{1/2} \varphi_n\rightarrow (-\Delta_\sigma)^{1/2}f$ in $L^p(\S)$ (take for example $ \varphi_n=P_{r_n} f$ where $P_r $ is the Poisson transform in the sphere (see \eqref{Poisson_operador} in Section \ref{section_3}) and $r_n\nearrow 1$). By \eqref{bakry}  
$\nabla_\sigma\varphi_n$ is a Cauchy sequence in $\bm{L}^p(\S)$ converging to a vector field $\bm{F} \in \bm{L}^p(\S)$. Conversely if $\varphi_n\in C^\infty(\S)$ converges to $f$ in $L^p(\mathbb{S}) $ and $\nabla_{\sigma}\varphi_n\rightarrow \bm{F}$ in $\bm{L}^p(\mathbb{S})$ norm,  then  by \eqref{bakry}  $(-\Delta_\sigma)^{1/2}\varphi_n$ converges in $L^p(\mathbb{S})$ and converges to $(-\Delta_\sigma)^{1/2} f$ in $\dist$, implying that
$(-\Delta_\sigma)^{1/2} f\in L^p(\S)$. It follows that \eqref{Bakry en Sobolev}  holds  for $ f\in W^{1,p}(\S)$.
\end{proof}

 \begin{lemma}\label{suma de grads esf}
 Let $g\in W^{1,p}(\S)$, $1<p<\infty$,  with $g=\sum_{\ell \geq 0} \sum_{|m| \leq \ell} a_{\ell, m}Y_{\ell, m}$ in $\dist$. Then
 \begin{equation*}
     \gradesf g=\sum_{\ell \geq 0} \sum_{|m| \leq \ell} a_{\ell, m}\gradesf Y_{\ell, m}\quad \text{in }\bm{D}'(\mathbb{S}).
 \end{equation*}
 \end{lemma}
\begin{proof}
    First notice that the lemma  is equivalent to prove that 
    \begin{equation*}
  \langle  \gradesf g,\bm{h}\rangle=\sum_{\ell \geq 0} \sum_{|m| \leq \ell} a_{\ell, m} \langle\gradesf Y_{\ell, m},\bm{h}\rangle.
\end{equation*} for any $\bm{h}\in \bm{C}^\infty(\S)$.
    
    If $g\in C^\infty(\S)$ then $g=\sum_{\ell \geq 0} \sum_{|m| \leq \ell} a_{\ell, m}Y_{\ell, m}$ in $C^\infty(\S)$ and  by  \eqref{bakry}
    \begin{equation*}
    \Vert \gradesf g-   \sum_{\ell =0}^L \sum_{|m| \leq \ell} a_{\ell, m}\gradesf Y_{\ell, m}\Vert_{\bm{L}^p}\leq  C\Vert (-\Delta_\sigma)^{1/2} (g-   \sum_{\ell =0}^L \sum_{|m| \leq \ell} a_{\ell, m} Y_{\ell, m})\Vert_{L^p}\rightarrow 0
    \end{equation*}
    and the lemma follows in this case.
Now let    $g\in W^{1,p}(\S)$ and a sequence $g_k\in C^\infty(\S)$ converging to $g$  in $W^{1,p}(\S)$.  Then

\begin{equation}\label{ksuma grads}
    \gradesf g_k\rightarrow \gradesf g \text{ in } \bm{L}^p(\S).
\end{equation}
Let
\begin{equation*}
   g_k=\sum_{\ell \geq 0} \sum_{|m| \leq \ell} a_{\ell, m}^kY_{\ell, m},
\end{equation*}
and 
\begin{equation*}
   \gradesf g_k=\sum_{\ell \geq 0} \sum_{|m| \leq \ell} a_{\ell, m}^k\gradesf Y_{\ell, m}.
\end{equation*}
Also, since $g_k\rightarrow g$ in $L^p(\S)$, then $lim_{k\rightarrow\infty} a_{\ell, m}^k=a_{\ell, m}.$

Now let $\bm{h}\in \bm{C}^\infty(\S)$ a vector field. Denote $\bm{h}_T$ the tangential component of $\bm{h}$, namely
\begin{equation*}
    \bm{h}_T(x')=\pi_{x'^\perp}\bm{h}(x')\quad x'\in \S,
\end{equation*}
where $\pi_{x'^\perp} $ is the orthogonal projection of $\bm{h}$ in the tangent space of $\bm{h}$ at $x'$.

Let \begin{equation*}
    I_k=\sum_{\ell > 0} \sum_{| m | \leq \ell} a_{\ell, m}^k\int_\S\gradesf Y_{\ell, m}\cdot \bm{h} d\sigma
=\sum_{\ell > 0} \sum_{| m | \leq \ell} a_{\ell, m}^k\int_\S\gradesf Y_{\ell, m}\cdot \bm{h}_T d\sigma
\end{equation*}

and
\begin{equation*}
    I=\sum_{\ell \geq 0} \sum_{| m | \leq \ell} a_{\ell, m}\int_\S\gradesf Y_{\ell, m}\cdot \bm{h} d\sigma=\sum_{\ell \geq 0} \sum_{| m | \leq \ell} a_{\ell, m}\int_\S\gradesf Y_{\ell, m}\cdot \bm{h}_T d\sigma.
\end{equation*}

Then by the divergence theorem in $\S$,

\begin{equation*}
  I_k=  -\sum_{\ell > 0} \sum_{| m | \leq \ell} a_{\ell, m}^k\int_\S Y_{\ell, m} \,div_\sigma \bm{h}_T d\sigma,
\end{equation*}

\begin{equation*}
    = -\sum_{\ell > 0} \sum_{| m | \leq \ell} \frac{a_{\ell, m}^k}{\ell(\ell+1)}\int_\S Y_{\ell, m} \,(-\Delta_\sigma)\, div_\sigma \bm{h}_T d\sigma
\end{equation*}
and

\begin{equation*}
    I=-\sum_{\ell > 0} \sum_{| m | \leq \ell} \frac{a_{\ell, m}}{\ell(\ell+1)}\int_\S Y_{\ell, m} \,(-\Delta_\sigma)\, div_\sigma \bm{h}_T d\sigma,
\end{equation*}
where $div_\sigma$ denotes the divergence in the sphere.
By \eqref{estima ylm} we have that $$\left| a_{\ell, m}^k\right|=\left|\int_\S g\,Y_{\ell,m}d\sigma\right| =O(\ell^{1/2} ).$$
Then $I_k$ converges to $I$ by the dominated convergence theorem. This together with \eqref{ksuma grads} implies

\begin{equation*}
  \langle  \gradesf g,\bm{h}\rangle=\sum_{\ell \geq 0} \sum_{|m| \leq \ell} a_{\ell, m} \langle\gradesf Y_{\ell, m},\bm{h}\rangle.
\end{equation*}
\end{proof}

\section{Elastic Poisson kernel in the unit ball  and boundary functions}\label{section_3}

Following Kupradze (see \cite[Ch. XIV, 1.26]{Ku}),  
the elastic Poisson kernel is the matrix-valued function 
$$\mathcal{\bm{P}}_e: \mathbb{B}  \times   \mathbb{S}  \rightarrow M_{3\times 3}(\mathbb{R}),$$ which can be seen as a perturbation of the Poisson kernel for harmonic functions, $P(x,\eta)$,  in $\mathbb{B}$, that is, 
\begin{equation}\label{ja_poisson_1}
 \mathcal{\bm{P}}_e(x,\eta)= P(x,\eta)\mathcal{I}+ \mathcal{L}(x,\eta),
\end{equation}
where $\mathcal{I}$ is the identity matrix,
\begin{equation}\label{ja_poisson_2}
P(x,\eta)=\frac{1}{4\pi}\frac{1-\vert x\vert^2}{\vert x-\eta\vert^3},
\end{equation}
and 
\begin{equation}\label{ja_poisson_3}  
  \mathcal{L}_{ij}(x,\eta)=\beta(1-\vert x\vert^2)\frac{\partial^2 \Phi}{\partial x_i\partial x_j}(x,\eta), \;\; i,j \in \{1,2,3 \},
  \end{equation}
with
\begin{equation}\label{ja_poisson_4}
  \Phi(x,\eta)=\frac{1}{4\pi}\int_0^1\left(\frac{1-t^2\vert x\vert^2}{(1-2tx\cdot \eta+t^2\vert x\vert^2)^{3/2}}-1-3tx\cdot\eta\right)\frac{dt}{t^{1+\alpha}}
  \end{equation}and 
  \begin{equation*}
  \beta=\frac{\lambda+\mu}{2\lambda+6\mu},\qquad \alpha=\frac{\lambda+2\mu}{\lambda+3\mu}.
   \end{equation*} 

   Notice that for $x$ and $\eta$ fixed, we have $4 \pi P(tx,\eta)-1-3tx\cdot\eta=O(t)$ so, since $\alpha \in (0,1)$,   $\Phi$ is well defined. Furthermore, by (\ref{acotacion_poisson_e_3}) and (\ref{acotacion_poisson_e_2}) of the appendix, $\frac{\partial^2 \Phi}{\partial x_i\partial x_j}(x,\eta)$ is also well defined.
   
The Poisson transform $P$ is defined by
\begin{equation}\label{Poisson_operador}
    Pf(x)=\int_{\S} P(t \xi, \eta) f(\eta) d\sigma  (\eta), \;\;\; f \in L^1(\S), \;\; x=t\xi\in\B.  
   \end{equation}

   We define for $ t  \in (0,1)$   
   \begin{equation}\label{Poisson_operador1}
    P_tf(\xi)=Pf(t\xi).    
   \end{equation}

It is known that, (see for example \cite[page 99]{Axler}),
\begin{equation}\label{Poisson_armonicos}
P(t \xi,\eta)=\sum_{\ell \geq 0} \sum_{|m| \leq \ell} t^\ell Y_{\ell, m}(\xi)  Y_{\ell, m}(\eta),     
\end{equation}
which allows us to see $P_t$ as a zonal multiplier associated to the sequence $\{t^\ell \}_{\ell =0,1,2, \cdot }$.  It is well known that $P_t$ is bounded in $L^p(\mathbb{S})$ for $1 \leq p \leq \infty $. 

The Hardy-Littlewood maximal function of $f \in L^p(\S)$, denoted by ${M}[f]$, is the function on $\S$ defined by
\begin{equation*}
 {M}   [f](\xi) = \sup_{\delta >0} \frac{1}{\sigma (\kappa (\xi, \delta))} \int_{\kappa (\xi, \delta)} |f(\eta)| d \sigma (\eta), 
\end{equation*}
where $\kappa (\xi, \delta)= \{ \eta \in \S: \; |\eta - \xi|< \delta \}$.

${M}$ is bounded on $L^p(\S)$ for $p \in (1, \infty)$ and weak type $(1,1)$.
It is a standard fact  ( see \cite[Ch VI, Sect 1]{Stein}) that if $f\in L^1(\S)$ then

\begin{equation}\label{ntmaximal vs maximal}
    \mathcal{N}(Pf)(\zeta)\leq {M}[f](\zeta),\; \;  \zeta\in\S.
\end{equation}
The elastic Poisson transform  $\bm{P}_e$ is defined by

\begin{equation*}
    \bm{P}_e \bm{f}(x)=\int_{\S} \mathcal{P}_e (x, \eta)  \bm{f} (\eta) d \sigma (\eta),
     \;\;\; \bm{f} \in \bm{L}^p(\S), \;\; 1 \leq p \leq \infty.
\end{equation*}
We have that 
\begin{align*}
    \bm{P}_e \bm{f}(x) &= P\bm{f}(x)+\bm{L}\bm{f}(x) \\ 
    &=  \int_{\S} P(x, \eta)  \bm{f} (\eta) d \sigma (\eta)+\int_{\S}  \mathcal{L}(x, \eta)\bm{f} (\eta) d \sigma (\eta).    
   \end{align*}

   
   For $\bm{\mu}\in \bm{M}(\S), $  we define likewise 
  \begin{equation*}
 \bm{P}_e\bm{\mu}(x)= \int_\S \mathcal{P}_e(x,\eta)d \bm{\mu}(\eta)=\int_\S \mathcal{P}_e(x,\eta)(d \mu_1(\eta),d \mu_2(\eta),d \mu_3(\eta))^t,
  \end{equation*}
  where $\bm{\mu}=(\mu_1,\mu_2,\mu_3).$


   \begin{proposition}\label{estimaciones} Let $\bm{f} \in \bm{L}^1(\S)$. We have
   \begin{itemize}
   \item[a)] The matrix function $ \partial^\alpha_x\mathcal{P}_e(x,\eta)$ is uniformly bounded in $K\times\S$ for any compact
   $K\subset \B$ and every partial derivative $\partial^\alpha_x.$
   \item[b)]$\vert \bm{L} \bm{f}(x)\vert \lesssim  (1-\vert x\vert^2)\int_0^1\frac{ P  |\bm{f}|(tx)}{(1-t^2\vert x\vert^2)^2}t^{1-\alpha}dt$.

   \item[c)] 
   $\max (\bm{P}_e \bm{f})(\xi)\lesssim \max  (P\vert \bm{f}\vert )(\xi)\lesssim {M} [|\bm{f}|](\xi).$
   \end{itemize}

   \end{proposition}
   
   \begin{proof}
   See the appendix.
   \end{proof}

   We have the standard result on Hardy spaces for the case of the  Lam\'e equation:
   \begin{theorem} 
   
   \begin{enumerate}\label{Clasico}
   \item  Let $1 < p \leq \infty$. $\bm{u} \in  \bm{h}_e^p(\B)$ if and only if $\bm{u}(x) = \bm{P}_e \bm{f}(x)$ for some $ \bm{f} \in \bm{L}^p (\S)
 $. 
   \item $\bm{u}\in \bm{h}_e^1 (\B)$ if and only if $\bm{u}(x)=\bm{P}_e \bm{\mu }(x)$ for some $\bm{\mu}\in \bm{M}(\S)$. 
   \end{enumerate}
   \end{theorem}
   \begin{proof}
   \textit{1.} Suppose that $\bm{u}\in \bm{h}_e^p(\B)$, with $p>1$. We proceed in the standard way: let $\bm{u}_r(\eta) =\bm{u} (r\eta)$, $\eta \in \mathbb{S}$ and $r \in (0,1)$. 
   Notice that since $\Delta_x^* \bm{u}  (rx)=0$ on $\overline{\B}$, then $\bm{u}  (rx)=\int_\S \mathcal{P}_e(x,\eta) \bm{u} _r(\eta)d\sigma (\eta)$.
   
    We have that $\mathcal{A}=\lbrace \bm{u}_r\rbrace_{r \in (0,1)}$ is a bounded set in $\bm{L}^p(\S)=\bm{L}^q(\S)^*,$ where $q$ is the conjugate exponent of $p$ with the duality
   $$\langle \bm{f},\bm{g}\rangle=\int_\S \bm{f}(\theta)\cdot \bm{g}(\theta)  d\sigma (\theta),
\;\;\; \bm{f}\in \bm{L}^p(\S), \; \bm{g}\in \bm{L}^q(\S).$$ 
   Then $\mathcal{A}$ is a $w^*$ relatively compact by the Banach-Alouglou theorem.  Hence there exist a sequence $r_n\nearrow 1$ and $\bm{f}\in \bm{L}^p(\S)$ such that $\mathbf{u}_{r_n}\rightarrow \bm{f}$ in the weak$^*$ topology. 
Since the rows of $\mathcal{P}_e(x,\cdot)$ are bounded functions (see Proposition \ref{estimaciones}) in $\S$ and hence they are in $ \bm{L}^q(\S)$,  it follows  that 
\begin{equation*}
\bm{u}(x)=\lim_{n\rightarrow\infty}\bm{u}(r_nx)=\lim_{n\rightarrow\infty}\int_\S \mathcal{P}_e(x,\eta)\bm{u}_{r_n}(\eta)d\sigma (\eta)=\int_\S \mathcal{P}_e(x,\eta)\bm{f}(\eta)d\sigma (\eta).
\end{equation*}
   
  Conversely, let $\bm{f}\in \bm{L}^p(\S)$. Then the standard properties of the Poisson kernel imply that 
  \begin{equation*}
  \Vert (P \bm{f})_r\Vert_{\bm{L}^p}\lesssim \Vert \bm{f}\Vert_{\bm{L}^p}.
  \end{equation*}
  Also, by Proposition \ref{estimaciones} b), if $\eta\in \S$,
  \begin{equation*}
 \vert (\bm{L} \bm{f})_r(\eta)\vert  \lesssim  (1-r^2)\int_0^1\frac{ P|\bm{f}|(tr\eta)}{(1-t^2r^2)^2}t^{1-\alpha}dt.
  \end{equation*}
so that using   Minkowskii's inequalty, 
\begin{equation*}
\Vert(\bm{L}\bm{f})_r\Vert_{\bm{L}^p}\lesssim(1-r^2)\int_0^1\frac{\Vert(P|\bm{f}|)_{rt}\Vert_{\bm{L}^p(\S)}}{(1-t^2r^2)^2}t^{1-\alpha}dt
\end{equation*}
\begin{equation*}
 \leq \Vert \bm{f}\Vert_{\bm{L}^p}\int_0^1\frac{(1-r^2)}{(1-t^2r^2)^2}t^{1-\alpha}dt\lesssim\Vert \bm{f}\Vert_{\bm{L}^p},
\end{equation*}
and this holds for $1\leq p<\infty$. This completes the proof of \textit{1.}  

Now we prove the representation of elements in $\bm{h}_e^1(\B)$. This is done  as in the previous case  using  the duality $\bm{C}(\S)^*=\bm{M}(\S)$, where $\bm{C}(\S)$ is the Banach space of continuous functions in $\S$ with values in $\R^3$. Finally, we will prove that $\bm{P}_e\bm{\mu}\in \bm{h}_e^1(\B)$ for $\bm{\mu}\in\bm{M}(\S).$ 
First recall that if $\mu$ is a Borel signed measure in $\S$ then
\begin{equation}\label{acota integral}
\left| \int_\S \varphi d\mu\right|\leq \int_\S \vert \varphi\vert d\vert \mu\vert,
\end{equation}
for every $\varphi\in C(\S)$. 
This implies that for any $0\leq r< 1$

\begin{equation*}
    \vert (P\bm{\mu})_r\vert \leq (P\vert \bm{\mu}\vert)_r.
\end{equation*}
Hence 
\begin{equation}\label{controla P}
    \Vert (P\bm{\mu})_r\Vert_{{\bm{L}^1}} \leq \Vert (P\vert \bm{\mu}\vert)_r\Vert_{L^1}\leq \vert \bm{\mu}\vert(\S),
\end{equation} 
by the theory of harmonic Hardy spaces. 
Similarly by \eqref{cota L} and \eqref{acota integral} we have
\begin{equation}\label{controla L}
\Vert(\bm{L}\bm{\mu})_r\Vert_{\bm{L}^1}\lesssim(1-r^2)\int_0^1\frac{\Vert(P\bm{\mu})_{rt}\Vert_{\bm{L}^1}}{(1-t^2r^2)^2}t^{1-\alpha}dt\leq \vert \bm{\mu}\vert(\S).
\end{equation}

Then,  by \eqref{controla P} and \eqref{controla L} we have that $\bm{P}_e\bm{\mu}\in \bm{h}^1_e(\B).$

   \end{proof}
  \begin{remark}
  By the open mapping theorem  $\bm{L}^p(\S)\cong \bm{h}_e^p(\B)$ for $1 < p \leq \infty$ and 
   $\bm{M}(\S)\cong \bm{h}_e^1(\B)$ via the elastic Poisson transformation $\bm{P}_e$.
  \end{remark}

    \begin{theorem}\label{clasico-maximal}
  $\bm{h}_e^p(\B)=\bm{H}_e^p(\B)$ for $1<p<\infty$.
  \end{theorem}
  \begin{proof}
  Let $\bm{u}\in \bm{h}_e^p(\B)$, and $\bm{f} \in \bm{L}^p(\S)$ such that $\bm{u}(x)=\bm{P}_e \bm{f}(x)$.
  Then by proposition \ref{estimaciones},
  \begin{equation*}
  \vert \bm{u}(x)\vert\leq C\mathcal{M}(|\bm{f}|)(\eta)  
  \end{equation*}
  for $x\in\Gamma_\eta$.   Then since ${M}[|\bm{f}|](\eta)\in \bm{L}^p(\S)$, we have that 
  \begin{equation*}
  \max \bm{u}\in L^p(\S).
  \end{equation*}
  
   \end{proof}
   
   
  We say that $\bm{u}:\B\rightarrow\R^3$  has non-tangential  limits almost everywhere if the limit 
  
  $$\lim_{x \in \Gamma_\xi, \; x \rightarrow \xi} \bm{u}(x).$$  exists for almost all $\xi\in\S$.

  We have the Fatou theorem for $\bm{h}_e^p(\B)$. (See \cite{MMMM} for very general results on Fatou theorems for elliptic systems).  
  \begin{theorem}\label{Fatou}
  For any $\bm{f}\in \bm{L}^1 (\S)$, $\bm{u}=\bm{P}_e \bm{f}$ converges non-tangentially to $\bm{f}$  almost everywhere. In particular this holds for every $\bm{u}\in \bm{h}_e^p(\B)$, $1 <  p \leq \infty$. Moreover, every  $\bm{u}\in \bm{h}_e^1(\B)$ has non-tangential limits almost everywhere.
  \end{theorem}
  
  \begin{proof} 
  First notice that if $\bm{f}\in \bm{C}(\S)$, then by elliptic regularity we have that $\bm{P}_e(\bm{f})\in \bm{C}(\overline{\B})$, 
We start with  $\bm{f}=(f_1,f_2,f_3)\in \bm{C}^\infty(\S)$.
 On the other hand, by Proposition \ref{estimaciones}c),  $\max \bm{P}_e\bm{f}$ is weak $(1,1)$.  
This together with the fact that $\bm{C}(\S)$ is dense in $\bm{L}^1 (\S)$,  implies that $\bm{P}_e\bm{f}$ has non-tangential limits  almost everywhere for  $\bm{f}\in \bm{L}^1 (\S)$. 
 

  Now,  to handle the case $\bm{u}\in \bm{h}_e^1(\B)$,  consider first  $\bm{\mu}\in \bm{M}(\S)$,     singular  with respect to the Lebesgue measure $\sigma$.  Then $\vert\bm{\mu}\vert$ is also singular so that (see \cite[Th. 6.42]{Axler}) $$\lim_{x \in \Gamma_\xi, \; x \rightarrow \xi} P\vert\bm{\mu}\vert(x)=0,$$   
  for $\xi$ in a measurable set $A$ with $\sigma(\S-A)=0$.  We claim that the same holds  for $\bm{P}_e\bm{\mu}$.
  
  Since \begin{equation*}
  \vert \bm{P}_e\bm{\mu}(x)\vert\leq \vert P\bm{\mu}(x)\vert+\vert \bm{L}\bm{\mu}(x)\vert,
  \end{equation*}
  it remains to prove that 
  \begin{equation}\label{L a cero}
      \lim_{x \in \Gamma_\xi, \; x \rightarrow \xi} \bm{L}\bm{\mu}(x)=0.
  \end{equation}

In fact, let $\epsilon>0$,  $\xi\in A$ and $0<\delta<1$ such that $P\vert \bm{\mu}\vert (y)<\epsilon$ for $y\in \Gamma_\xi$ with $\vert \xi-y\vert<\delta$.
We write 
\begin{equation*}
\vert \bm{L}\bm{\mu}(x)\vert\lesssim (1-\vert x\vert^2)\left(\int_0^{\delta/2}+\int_{\delta/2}^1\right)
\frac{P\vert\bm{\mu}\vert(tx)}{(1-\vert tx\vert^2)^2}t^{1-\alpha}dt.
\end{equation*}
 Notice that if $x\in \Gamma_\xi$ with  $\vert x-\xi\vert<\delta/2$ then $tx\in\Gamma_\xi$ and $\vert tx-\xi\vert<\delta$ for  $0<1-t<\delta/2$. 
  This implies by \eqref{integral} that 
  \begin{equation*}
  (1-\vert x\vert^2)\int_{\delta/2}^1 \frac{P\vert\bm{\mu}\vert(tx)}{(1-\vert tx\vert^2)^2}t^{1-\alpha}dt\leq C\epsilon.
  \end{equation*}
 On the other side, since $P\vert\bm{\mu}  \vert(tx) $ is bounded for $t\in[0,\delta]$ and $x\in \B$, it is clear that 
  \begin{equation*}
  (1-\vert x\vert^2)\int_0^{\delta/2} \frac{P\vert\bm{\mu}\vert(tx)}{(1-\vert tx\vert^2)^2}t^{1-\alpha}dt\rightarrow 0 
  \end{equation*}
  as $x\rightarrow \xi $.  Hence \eqref{L a cero} holds.

  Now if $\bm{\mu}\in \bm{M}(\S)$ then by the Radon Nikodym theorem 	$ d\bm{\mu}=fd\sigma+d\bm{\mu}_s$, $f\in L^1(S)$ and $\bm{\mu}_s$ singular,  so that the Theorem follows from the previous analysis.

  \end{proof}
  \begin{remark}
   Notice that if $\bm{u}\in \bm{H}^1_e(\B),$ then since $\mathcal{N} \bm{u}\in\bm{L}^1 (\S)$, its boundary  limit $\bm{f}$ according to Theorem \ref{Fatou} belongs to $\bm{L}^1 (\S)$ by the dominated convergence theorem and $\bm{u}_r\rightarrow \bm{f}$ as $r\rightarrow 1-$ in   $\bm{L}^1 (\S)$, hence  $\bm{u}=\bm{P}_e \bm{f}$.
   \end{remark}

   Recall that $\bm{u}$ is a Riesz field such that  $div\, \bm{u}=0$ and $\nabla\times  \,\bm{u}=0$. This is equivalent to $\bm{u}=\nabla h$, where $h$ is a harmonic function, and also to the condition $\nabla \bm{u}$ symmetric with null trace.

\vspace{.5 cm}
{\it Riesz fields and solutions of $\Delta^*u=0$ independent of $\mu$ and $\lambda$.} There are fields like  $\bm{u}(x)=x$, that are solutions of the Lamé equation for any value of $\lambda$ and $\mu$.   In general, any harmonic field $\bm{u}:\B\rightarrow \R^3$ such that $\nabla div\, \bm{u}=0$ is a solution to the Lam\'e equation for any $\mu$ and $\lambda$. An example of this are the Riesz fields.

   If $\bm{u}$ is a Riesz field then $\Delta \bm{u}=0$ and $\nabla div \, \bm{u}=\nabla\times ( \nabla\times \bm{u})=0$, hence $\Delta^* \bm{u}=0$    
 for any eligible $\lambda, \mu$.
An easy calculation shows that if $\bm{u}$ is a Riesz field then $\bm{v}(x)=x\times \bm{u}(x)$ is harmonic and $div\, \bm{v}=0.$ So that it is a solution of the Lam\'e equation for any $\lambda$ and $\mu$.
     
 Notice  that if $\bm{u}$  is a solution of $\Delta^*\bm{u}=0$ for all $\mu$ and $\lambda$ and has boundary value $\bm{f}$, then 
 \begin{equation}\label{P_e=P}
     \bm{u}=\bm{P}_e\bm{f} =P\bm{f}.
 \end{equation}

  \section{Vector spherical harmonics and vector-valued $\bm{L}^p$ spaces} 

 For $\ell \geq 0$ we define the real vector functions on $\S$ 
\begin{align}
 \bm{E}_{\ell,m}^{+}(\eta) &=(\ell+1) \bm{Y}_{\ell,m}^\vee (\eta)-\gradesf Y_{\ell,m}(\eta) ,&  |m| \leq \ell,  \nonumber \\
 \bm{E}_{\ell,m}^{-}(\eta ) & =\ell \bm{Y}_{\ell,m}^\vee (\eta)+\gradesf Y_{\ell,m} (\eta), &  |m| \leq \ell,  \nonumber \\
 \bm{E}_{\ell,m}^{0}(\eta) &=\eta\times \gradesf Y_{\ell,m}(\eta), & |m| \leq \ell, \label{E_ell_m_+}
     \end{align}
and the following vector spaces  of vector fields on $\mathbb{S}$
  \begin{align}
  \bm{E}^+_\ell (\S) &=\bigvee\left\lbrace \bm{E}_{\ell,m}^{+}, \; |m|\leq \ell     \right\rbrace
  , \nonumber\\
   \bm{E}^-_\ell (\S) & =\bigvee\left\lbrace \bm{E}_{\ell,m}^{-}, \;  \; |m| \leq \ell    \right\rbrace
  , \nonumber \\
   \bm{E}^0_\ell (\S)& =\bigvee\left\lbrace \bm{E}_{\ell,m}^{0},\;   \; |m| \leq \ell     \right\rbrace .  \label{E_ell^+}  
  \end{align}
  We note that $\bm{E}^-_0 (\S)=\bm{E}^0_0 (\S)=\{0\}$.
  
  It can be checked that for $m, m' \in \{-\ell, \cdots , \ell\}$, (see \cite{Freeden}),
  \begin{align}
      \int_\S \bm{E}_{\ell,m}^{+}(\eta) \cdot \bm{E}_{\ell,m'}^{+}(\eta) d \sigma (\eta)& =(\ell +1) (2 \ell +1)\delta_{m m'} , \nonumber \\
      \int_\S \bm{E}_{\ell,m}^{-}(\eta) \cdot \bm{E}_{\ell,m'}^{-}(\eta) d \sigma (\eta)& =\ell  (2 \ell +1)\delta_{ m m'} , \;\; \ell \neq 0, \nonumber \\
      \int_\S \bm{E}_{\ell,m}^{0}(\eta) \cdot \bm{E}_{\ell,m'}^{0}(\eta) d \sigma (\eta)& =\ell  ( \ell +1)\delta_{ m m'} , \;\; \ell \neq 0, \label{ortogonalidad_juan}
  \end{align}
  
  If $\mu_\ell^+= \sqrt{(\ell +1) (2 \ell +1)}$, $\mu_\ell^-= \sqrt{\ell  (2 \ell +1)}$ and $\mu_\ell^0= \sqrt{\ell  ( \ell +1)}$, then
  $\left\{\frac{1}{\mu_\ell^+} \bm{E}_{\ell,m}^{+} \right\}_{|m| \leq \ell}$, $\left\{\frac{1}{\mu_\ell^-} \bm{E}_{\ell,m}^{-} \right\}_{ |m| \leq \ell}$ and $\left\{\frac{1}{\mu_\ell^0} \bm{E}_{\ell,m}^{0} \right\}_{ |m| \leq \ell}$ are an orthonormal basis of $\bm{E}^+_\ell (\S)$, $\bm{E}^-_\ell (\S)$  and $\bm{E}^0_\ell (\S)$ respectively.
\begin{theorem}\label{descom orto}
 The spaces $\bm{E}^+_\ell(\S)$, $\bm{E}^-_\ell(\S)$  and $\bm{E}^0_\ell (\S)$ are mutually orthogonal and their elements  are the restriction of vector harmonic homogeneous polynomials of degree $\ell+1$, $\ell-1$ and  $\ell$ respectively, (in the latter two cases $\ell \geq 1$ ).

We have the orthogonal decomposition 
 \begin{equation}\label{L2Decom}
 \bm{L}^2(\mathbb{S}) =\bm{L}^2_+(\mathbb{S})\oplus \bm{L}^2_-(\mathbb{S})\oplus \bm{L}^2_0(\mathbb{S}),
\end{equation}
where
\begin{align}\label{gradientes}
 \bm{L}^2_+(\mathbb{S}) &:=\bigoplus_{\ell \geq 0}\bm{E}_{\ell }^+(\S),   \\
 \bm{L}^2_-(\mathbb{S}) & :=\bigoplus_{\ell \geq 1}\bm{E}_{\ell }^- (\S), \\
 \bm{L}^2_0(\mathbb{S}) & :=\bigoplus_{\ell \geq 1 } \bm{E}_{\ell }^0 (\S).
\end{align}
\end{theorem}
\begin{proof} Let $\ell, \ell' \geq 0, \; \ell \neq \ell' \neq 0$, $|m| \leq \ell$ and $|m'| \leq \ell'$. 
\begin{align*}
  \langle \bm{E}_{\ell,m}^{+},\bm{E}_{\ell',m'}^{-} \rangle &= \ell' (\ell +1)\langle Y_{\ell,m},Y_{\ell',m'}\rangle- \langle  {\nabla}_\sigma Y_{\ell,m},{\nabla}_\sigma Y_{\ell',m'} \rangle \\
  & =\ell' (\ell +1)\delta_{\ell \ell'} \delta_{mm'} +\int_\S Y_{\ell ,m} (\eta) \Delta_\sigma Y_{\ell',m' } d \sigma (\eta)  \\
  & =\ell' (\ell +1)\delta_{\ell \ell'} \delta_{mm'}-\ell' (\ell'+1)\int_\S Y_{\ell ,m} (\eta)  Y_{\ell',m' } d \sigma (\eta) \\
   & =\ell' (\ell +1)\delta_{\ell \ell'} \delta_{mm'}-\ell' (\ell'+1) \delta_{\ell \ell'} \delta_{mm'}=0,
\end{align*}
which proves that the spaces $\bm{E}^+_\ell (\S)$ and $\bm{E}^-_{\ell'} (\S)$ are orthogonal. It is clear that $\bm{E}^0_{\ell'} (\S)$  is orthogonal to $\bm{E}^+_\ell (\S)$ and $\bm{E}^-_{\ell} (\S)$.

Now we show  that the elements of the spaces $\bm{E}^+_\ell(\S)$, $\bm{E}^-_{\ell}(\S)$ and $\bm{E}^0_{\ell}(\S)$  are the restriction of vector harmonic homogeneous polynomials of degree $\ell+1$, $\ell-1$ and  $\ell$ respectively.
 Let $ Y\in\mathcal{H}^\ell$,  then by  \eqref{grad arm esf} and for $x=|x|x'$ we have
 \begin{align}
  \left. \nabla Y(x) \right|_{x=x'} & =\ell \bm{Y}^\vee (x')+\gradesf Y(x'), \;\;  \ell \neq 0,    \nonumber  \\  
  \left. ((2\ell+1)Y(x)x-|x|^2\nabla Y(x))\right|_{x=x'} & = (\ell+1)\bm{Y}^\vee (x')-\gradesf Y(x'), \nonumber \\
  \left. x \times \nabla Y(x) \right|_{x=x'} &=  x'\times \nabla_\sigma Y(x'), \;\;  \ell \neq 0, \label{restriccion} 
 \end{align}
 and the result follows from the fact that   $(2\ell+1)Y(x)x-|x|^2\nabla Y(x),\, x \times \nabla Y(x)$  and  $\nabla Y(x)$ are vector harmonic homogeneous polynomials of degree $\ell +1, \ell$ and $\ell-1$ respectively.

 Finally to see \eqref{L2Decom} we note that 
 $$\bm{E}^+_{\ell-1}( \S)\oplus \bm{E}^-_{\ell+1}(\S)\oplus \bm{E}^0_\ell (\S) \subset (\mathcal{H}^{\ell})^3 $$
 and by considering the dimension of the spaces in the above  inclusion  we have that they are equal and hence \eqref{L2Decom}
 holds.

 \end{proof}

 \begin{remark} 
 
 The above theorem tells us that 
 $$\left\{\frac{1}{\mu_\ell^+} \bm{E}_{\ell,m}^{+} \right\}_{
    \ell \geq   0, \; |m| \leq \ell} ,  \left\{\frac{1}{\mu_\ell^-} \bm{E}_{\ell,m}^{-} \right\}_{
    \ell \geq   1, \; |m| \leq \ell} \textrm{ and } \left\{\frac{1}{\mu_\ell^0} \bm{E}_{\ell,m}^{0} \right\}_{
    \ell \geq   1, \; |m| \leq \ell},$$ are an orthonormal basis of $\bm{L}^2_+ (\S)$, $\bm{L}^2_- (\S)$  and $\bm{L}^2_0 (\S)$ respectively.

 \end{remark} 

  We will use this systems  throughout this section. We also notice that from \eqref{estima ylm} we have the estimate in $\ell$
\begin{equation}\label{estima Elm}
\Vert  \bm{E}_{\ell,m}^{\#} \Vert_{\bm{L}^\infty(\mathbb{S})}=O(\ell^{3/2}), \;\; \# \in \{ +,-,0\}.
\end{equation}

Denote by $\bm{\pi}_+, \bm{\pi}_-,\bm{\pi}_0$ the orthogonal projection of $ \bm{L}^2(\S)$ onto $ \bm{L}_+^2(\mathbb{S})$, $ \bm{L}_-^2(\mathbb{S})$ and  $ \bm{L}_0^2(\mathbb{S})$  respectively.
\begin{theorem}\label{desc Lp}
For any $1 < p<\infty$ the projections $\bm{\pi}_+, \bm{\pi}_-, \bm{\pi}_0$ defined on  
 $\bm{C}^\infty(\mathbb{S})$ can be extended to continuous projections in $\bm{L}^p(\mathbb{S})$, and
\begin{equation}\label{LpDecom}
\bm{L}^p(\mathbb{S})=\bm{L}_+^p(\mathbb{S})\oplus \bm{L}_-^p(\mathbb{S})\oplus \bm{L}_0^p(\mathbb{S}),
\end{equation}
where $\bm{L}_+^p(\mathbb{S})$, $\bm{L}_-^p(\mathbb{S})$ and $ \bm{L}_0^p(\mathbb{S})$ are the image of $\bm{L}^p(\mathbb{S})$ under $\bm{\pi}_+, \bm{\pi}_- \text{ and } \bm{\pi}_0$ respectively.
\end{theorem}

\begin{proof}
   See the appendix.
   \end{proof}

We define the operator acting on functions on $\S,$
\begin{equation}\label{operador L_-}
   \mathcal{L}_{-}g= (\bm{\mathcal{M}_\ell g}) ^\vee +\gradesf g .
\end{equation}

By \eqref{grad arm esf}, $$\nabla (PY)=P(\mathcal{L}_-Y), \;\;Y\in \mathcal{H}^\ell. $$

\begin{theorem}\label{espacio -1}
Let $1<p<\infty$. Then

\begin{itemize}
\item[a)] The operator $\mathcal{L}_-$ maps continuously $W^{1,p}(\mathbb{S})$ onto $\bm{L}_-^p(\mathbb{S})$ and its kernel are the constant functions. 
\item[b)] For $g \in W^{1,p}(\S)$ we have
\begin{equation}\label{PL}
 P(\mathcal{L}_-g)(x)= \nabla (Pg)(x)\qquad x\in\B. 
\end{equation}
\item[c)] $\bm{f}\in \bm{L}_-^p(\mathbb{S})$ if and only if there exists $ g  \in W^{1,p}(\mathbb{S})$ such that 
$$\bm{f}(\eta)= \lim_{r\rightarrow 1^-}    \nabla ( Pg)(r \eta).$$ 
\end{itemize}
 \end{theorem}
 \begin{proof} 
\textit{ a)} Let $g\in W^{1,p}(\mathbb{S})$  with expansion $g=\sum_{\ell \geq 0} \sum_{|m| \leq \ell} a_{\ell , m} Y_{\ell, m}$ in $\dist.$ 
Then \begin{equation*}
 \mathcal{M}_{\ell}g=\sum_{\ell \geq 0} \sum_{|m| \leq \ell} \ell a_{\ell,m}Y_{\ell,m}=\mathcal{M}_{\beta_\ell}(-\Delta_\sigma)^{1/2} g\in  L^p(\mathbb{S}),
 \end{equation*}
 where $\beta_\ell=\frac{\ell}{\sqrt{\ell(\ell+1)}},$ if $\ell>0$ and $\beta_0=0$,
 and by Theorem \ref{Strichartz} and Lemma \ref{lema_3} we have that
  $\mathcal{L}_-: W^{1,p}(\mathbb{S})\rightarrow \bm{L}^p(\mathbb{S})$ is continuous. 
 
Since  $\mathcal{L}_-Y_{\ell,m}= \bm{E}_{\ell,m}^{-}$, then if $g$ is smooth, 
 \begin{equation*}
 \mathcal{L}_-g= \sum_{\ell \geq 0} \sum_{|m| \leq \ell} a_{\ell,m}\bm{E}_{\ell,m}^{-}\in \bm{L}_-^p(\mathbb{S}).
 \end{equation*}
 
Finally approximating a $g\in W^{1,p}(\mathbb{S})$  by a sequence of smooth functions, it follows that $\mathcal{L}_-g\in \bm{L}_-^p(\mathbb{S}). $

Notice that for  $g \in  W^{1,p}(\mathbb{S})$,  $ (\bm{\mathcal{M}_{\ell}g})^\vee(x')$ and $\gradesf g(x')$ are orthogonal for almost every $x'\in \S. $ From this it follows that 
 $\mathcal{L}_-g=0$ if and only if  $\nabla_\sigma g=0$, namely if and only if  $g$ is a constant.

 Now we prove that $\mathcal{L}_{-}$ is onto.
 
  If $\bm{f}\in \lpm$ is smooth then we can write
 \begin{equation*}
 \bm{f}= \sum_{\ell \geq 1} \sum_{|m| \leq \ell} a_{\ell,m}\bm{E}_{\ell,m}^{-}.
 \end{equation*}
 Then $\bm{f}=\mathcal{L}_-g$, with $g=\sum_{\ell \geq 1} \sum_{|m| \leq \ell} a_{\ell,m}Y_{\ell,m}$. 
 
 Again by the orthogonality of $ (\bm{\mathcal{M}_{\ell}g})^\vee(x')$ and $\gradesf g(x')$ implies that $$\vert \mathcal{L}_- g(x')\vert=\vert (\bm{\mathcal{M}_{\ell}g})^\vee(x') +\gradesf g(x')\vert\sim \left(\vert (\bm{\mathcal{M}_{\ell}g})^\vee(x')\vert^p+ \vert\gradesf g(x')\vert^p\right)^{1/p},$$
 uniformly in $x'$, hence

\begin{equation*}
\Vert (-\Delta)^{1/2}g\Vert_{L^p}\lesssim\Vert\gradesf g\Vert_{\bm{L}^p}\lesssim\Vert \bm{f} \Vert_{\bm{L}^p}.
\end{equation*}
 Moreover, since $\int_{\S}g(x')dx'=0,$ we can write $g=\mathcal{\gamma_\ell}(-\Delta_\sigma)^{1/2}$, with
 $\gamma_\ell=\frac{\ell}{\sqrt{\ell(\ell+1)}},$ if $\ell>0$ and $\gamma_0=0.$

 Hence
 \begin{equation*}
 \Vert g\Vert_{L^p}\lesssim\Vert (-\Delta_\sigma)^{1/2}g\Vert_{L^p}\lesssim\Vert \bm{f} \Vert_{\bm{L}^p},
 \end{equation*}
 that is, \begin{equation*}
 \Vert g\Vert_{W^{1,p}}\lesssim\Vert \bm{f} \Vert_{\bm{L}^p}.
 \end{equation*}
 
 Now, if  $\bm{f}$ is any element of $\lpm$, let $\bm{f}_n\in\lpm \cap \bm{C}^\infty(\mathbb{S})$ such that $\bm{f}_n$ converges to $\bm{f}$ in $\bm{L}^p(\mathbb{S})$ and $g_n\in C^\infty(\mathbb{S})$ such that $\mathcal{L}_-  ( g_n)=\bm{f}_n$. The previous argument shows that
 \begin{equation*}
 \Vert g_n-g_m\Vert_{W^{1,p}}\leq C\Vert \bm{f}_n-\bm{f}_m \Vert_{\bm{L}^p},
 \end{equation*}
 and the limit $g$ of $g_n$ in $W^{1,p}(\S)$ satisfies $\mathcal{L}_-g=\bm{f}$.

\textit{b)} Let $g \in C^\infty (\S)$ such that $g=\sum_{\ell \geq 0} \sum_{|m| \leq \ell} a_{\ell,m}Y_{\ell,m}$ in $W^{1,p}(\S).$ For $x \in \mathbb{B}$ fixed  by \eqref{grad arm esf} (noticing that all the series below converge uniformly on compact subsets of $\B$) 
$$  P (\mathcal{L}_-g)(x)=\sum_{\ell \geq 0} \sum_{|m| \leq \ell} a_{\ell,m} P (\mathcal{L}_-Y_{\ell , m})(x)  =\sum_{\ell \geq 0} \sum_{|m| \leq \ell} a_{\ell,m} \nabla (P Y_{\ell , m})(x)$$
$$=\nabla \left(\sum_{\ell \geq 0} \sum_{|m| \leq \ell} a_{\ell,m} P Y_{\ell , m}\right)(x)= \nabla \left(P\sum_{\ell \geq 0} \sum_{|m| \leq \ell} a_{\ell,m}  Y_{\ell , m}\right)(x) =\nabla (P g)(x).$$
Hence the result holds in this case.
Now let  $g \in {W^{1,p}(\S)}$ and $g_n \in C^\infty (\S)$ such that $g_n \rightarrow g$ in $W^{1,p}(\S)$. If $x \in \B$ 
we have $P(\mathcal{L}_- g_n )(x) \rightarrow P(\mathcal{L}_- g )(x)$. On the other hand 
$\nabla (Pg_n)(x) \rightarrow \nabla (Pg)(x)$ and $\nabla (Pg_n)(x)=P(\mathcal{L}_-g_n)(x) \rightarrow  P( \mathcal{L}_-g)(x)$, then $\nabla (P g)= P (\mathcal{L}_-g)$.

\textit{c)} This follows directly from a) and b).
 
 \end{proof}

 Now we consider the operator acting on functions on $\S,$
\begin{equation}\label{operador L_0}
   \mathcal{L}_{0}g(\eta)= \eta \times \nabla_\sigma g (\eta) .
\end{equation}

 \begin{theorem} \label{espacio 0}
Let $1<p<\infty$. Then
\begin{itemize}
\item[a)] The operator $\mathcal{L}_0$ maps continuously $W^{1,p}(\mathbb{S})$ onto $\bm{L}_0^p(\mathbb{S})$ and its kernel are the constant functions. 
\item[b)] For $g \in W^{1,p}(\S)$ we have
\begin{equation}\label{PL_0}
   P (\mathcal{L}_0 g)(x)=x \times \nabla (Pg)(x), \;\; x \in \B.   
\end{equation}
\item[c)] $\bm{f}\in \bm{L}_0^p(\mathbb{S})$ if and only if there exists $ g  \in W^{1,p}(\mathbb{S})$ such that 
$$\bm{f}(\eta)= \lim_{r\rightarrow 1^-}  r \eta \times \nabla (Pg)(r \eta)  .$$ 
\end{itemize}

\end{theorem}

\begin{proof} 
\textit{ a)} Let $g \in W^{1,p}(\S)$ with expansion $g=\sum_{\ell \geq 0} \sum_{|m| \leq \ell} a_{\ell , m} Y_{\ell, m}$ in $\dist$ and $\tilde{g}=\sum_{\ell \geq 1} \sum_{|m| \leq \ell} a_{\ell , m} Y_{\ell, m}$. By using Lemma \ref{lema_3} we have
$$\| \mathcal{L}_0 g \|_{\bm{L}^p}=\| \nabla_\sigma \tilde{g} \|_{\bm{L}^p} \lesssim \| (- \Delta_\sigma)^{\frac{1}{2}} \tilde{g} \|_{L^p} \lesssim \| g \|_{W^{1,p}},$$
implying  that $\mathcal{L}_0$ maps continuously $W^{1,p}(\mathbb{S})$ onto $\bm{L}_0^p(\mathbb{S})$. 

Let us  see now that $\mathcal{L}_0$ is onto: let $\bm{f} \in \bm{L}^p_0(\S)$ be and $\bm{f}_n \in \bm{C}^\infty (\S) \cap \bm{L}^p_0(\S) $
converging to $\bm{f}$ in $\bm{L}^p(\mathbb{S})$. We have that
\begin{align*}
    \bm{f}_n (\eta) &= \sum_{\ell \geq 1} \sum_{|m| \leq \ell} a^n_{\ell, m} \bm{E}_{\ell,m}^{0}(\eta) = \sum_{\ell \geq 1} \sum_{|m| \leq \ell} a^n_{\ell, m} 
 \eta \times \nabla_\sigma Y_{\ell,m}(\eta)    \\  &= \eta \times \nabla_\sigma (g_n)= \mathcal{L}_0(g_n),
 \end{align*}
where $g_n (\eta)=\sum_{\ell \geq 1} \sum_{|m| \leq \ell} a^n_{\ell, m} 
  Y_{\ell,m}(\eta)  $. Since $\int_{\S}g_n(\eta)d\sigma(\eta)=0$, it follows that

$$\|g_n\|_{W^{1,p}} \sim \| \nabla_\sigma  g_n\|_{L^{p}}\lesssim\| \bm{f}_n \|_{\bm{L}^p},$$ proving that $g_n \in W^{1, p}(\S)$.

In a similar way 
$$\|g_n -g_m\|_{W^{1,p}} \lesssim  \| \bm{f}_n -\bm{f}_m \|_{\bm{L}^p},$$
so there is $g \in W^{1,p} (\S)$ such that $g_n$ converges to $g \in W^{1,p}(\S)$. As
$\mathcal{L}_0$ is continuous, $\mathcal{L}_0(g_n)=\bm{f}_n$ converges to $\mathcal{L}_0(g)$ and it must be verified that  $\mathcal{L}_0(g)=\bm{f}$.

Finally, we will study the kernel of $\mathcal{L}_0$. Let $g \in W^{1,p}(\S)$ be such that $\mathcal{L}_0 g(\eta)=\eta \times \nabla_\sigma g(\eta)=0$ implying  that $\nabla_\sigma g=0$ in $\bm{L}^p(\S)$, and  by Lemma  \ref{lema_3} $(-\Delta)^{\frac{1}{2}}g=0$ in $L^p(\S)$. If $g=\sum_{\ell \geq 0} \sum_{|m| \leq \ell} a_{\ell, m} Y_{\ell , m}$ in $\dist$, then $(-\Delta)^{\frac{1}{2}}g=\sum_{\ell \geq 1} \sum_{|m| \leq \ell} \ell (\ell +1) a_{\ell, m} Y_{\ell , m}=0$ in $\dist$. This implies that  $a_{\ell ,m}=0, \; \ell \geq 1$ and $|m| \leq \ell$ and therefore $g$ is a constant.

\textit{ b)} Note that if $g=Y,  \;\;Y\in \mathcal{H}^\ell$, then it follows from \eqref{restriccion} that
$P(\mathcal{L}_0 g)(x)= x \times \nabla Y(x)$. Then the proof of b) for $g\in W^{1,p}$ follows as the proof of part b) of Theorem \ref{espacio -1}. 

\textit{c)} This follows directly from a) and b).

\end{proof}

 Finally  we define the  operator acting on functions on S,
\begin{equation}\label{operador L_+}
   \mathcal{L}_{+}g= (\bm{\mathcal{M}_{\ell+1} g}) ^\vee -\gradesf g .
\end{equation}
 
 We have
\begin{theorem}\label{espacio +1}
Let $1<p<\infty$. Then

\begin{itemize}
\item[a)] The operator $\mathcal{L}_+$ is an isomorphism of  $W^{1,p}(\mathbb{S})$ onto $\bm{L}_+^p(\mathbb{S}).$ 

\item[b)] For $g \in W^{1,p}(\S)$ we have
\begin{equation}\label{PL_1}
   P (\mathcal{L}_+ g)(x)=[ 2 x \cdot \nabla (Pg)(x) + Pg(x)] x- |x|^2 \nabla (Pg)(x), \;\; x \in \B.   
\end{equation}

\item[c)]
$\bm{f}\in \bm{L}_+^p(\mathbb{S})$ if and only if there exists $ g  \in W^{1,p}(\mathbb{S})$ such that 
$$\bm{f}(\eta)= \lim_{r\rightarrow 1^-} \left(  [ 2 r \eta \cdot \nabla (Pg)(r \eta) + Pg(r \eta)] r \eta - r^2 \nabla (Pg)(r \eta)\right) .$$ 
\end{itemize}
 \end{theorem}

  \begin{proof} 
a) 
     The proof is analogous to the proof of Theorem \ref{espacio -1} a).
  
b) First consider    $g=\sum_{\ell \geq 0} \sum_{|m| \leq \ell} a_{\ell , m } Y_{\ell , m} \in C^\infty(\S)$. Define
$$ h(x)= Pg(x)=\sum_{\ell \geq 0} \sum_{|m| \leq \ell} a_{\ell , m }Y_{\ell , m}(x), \hspace{0.3cm } x \in \mathbb{B}.$$ then 
by \eqref{grad arm esf},
$$ x\cdot\nabla h(x)=\sum_{\ell \geq 0} \sum_{|m| \leq \ell} \ell a_{\ell,m} Y_{\ell,m}(x).$$
Hence  if we let \begin{equation*}
\bm{u}(x)=\{2x\cdot \nabla h(x)+h(x)\}x-|x|^2\nabla h(x),
\end{equation*}  
it is easy to see that $\bm{u}$ is harmonic on $\B$ and 
$$\bm{u}(x)=\sum_{\ell \geq 0} \sum_{|m| \leq \ell}\{(2 \ell+1)a_{\ell , m }Y_{\ell , m}(x)\}x-
\sum_{\ell \geq 0} \sum_{|m| \leq \ell}a_{\ell , m }\nabla Y_{\ell , m}(x) .$$
Furthermore by \eqref{grad arm esf} for $x=r\eta$ this simplifies to 
\begin{align*}
\bm{u}(r\eta) &=\sum_{\ell \geq 0} \sum_{|m| \leq \ell}a_{\ell , m }r^{\ell+1}\{( \ell+1)Y_{\ell , m}(\eta)\eta -\nabla_\sigma Y_{\ell , m}(\eta)\}  \\ & =\sum_{\ell \geq 0} \sum_{|m| \leq \ell}a_{\ell , m }r^{\ell+1}\mathcal{L}_{+}(Y_{\ell , m})(\eta),
\end{align*}
and the result follows in this case, noting that $\lim_{r\rightarrow 1^-}\bm{u}(r\eta)= \mathcal{L}_+ g(\eta).$

The proof  for $g\in W^{1,p}(\S)$ follows as the proof of part b) of Theorem \ref{espacio -1}. 

\textit{c)}  follows directly from \textit{a)} and \textit{b)}.
\end{proof}

The following result follows at once from  Theorems \ref{espacio -1},\ref{espacio 0} and \ref{espacio +1}.

 \begin{corollary}
  Let $1<p<\infty$. Then  $\bm{u}$ belongs to the harmonic vector-valued Hardy space 
   $\bm{h}^p(\B) $ if and only if  for some $g, \widetilde{g}, h \in W^{1,p}(\mathbb{S}) $
   \begin{equation*}
     \bm{u}(x)   =\nabla (Pg)(x)+x \times \nabla (P\widetilde{g})(x)    +[ 2 x \cdot \nabla \!(Ph)(x) + Ph(x)] x- |x|^2 \nabla (Ph)(x).   
   \end{equation*}

 \end{corollary}

\section{The spaces $\bm{h}_{+}^p(\B)$, $\bm{h}_{-}^p(\B)$ and $\bm{h}_{0}^p(\B)$}
In this section we study the elastic extensions of the vector $L^p$ spaces defined in section 5.

In \cite[lemma 10.3.6]{Freeden} it is proved the following
\begin{lemma}\label{easyP_e}
The solutions to the Dirichlet problem 
\begin{align}
\dest \bm{u}&=0 \ \ \text{ in }\mathbb{B},  \nonumber \\
\bm{u}&=\bm{f} \ \  \text{ in }\mathbb{S}, \nonumber 
\end{align}
are given as follows: 
\begin{itemize}
   \item[a)] If $\bm{f}=\bm{E}_{\ell,m}^{+}$ then $$\bm{u}_{\ell,m}^{+}(x)=(2\ell+1) \big{(}Y_{\ell,m}(x)x+\alpha_{\ell}(|x|^2-1)\nabla Y_{\ell,m}(x)\big{)}- \nabla Y_{\ell,m}(x), $$
   where $ \alpha_{\ell}=-\frac{(\ell+3)\tau+2}{2(\ell(\tau+2)+1)} $ and $\tau=\frac{\lambda + \mu}{\mu}.$ 
  \item[b)] If $\bm{f}=\bm{E}_{\ell,m}^{-}$ then $\bm{u}_{\ell,m}^{-}(x)=\nabla  Y_{\ell,m}(x). $
  \item[c)] If $\bm{f}=\bm{E}_{\ell,m}^{0}$ then $\bm{u}_{\ell,m}^{0}(x)=x\times \nabla Y_{\ell,m}(x). $
\end{itemize}
\end{lemma}

We will denote by $\bm{h}_{+}^p(\B)$, $\bm{h}_{-}^p(\B)$ and $\bm{h}_{0}^p(\B)$ the image of $ \bm{L}^p_+(\mathbb{S})$, $\bm{L}^p_-(\mathbb{S})$ and  $ \bm{L}^p_0(\mathbb{S})$ respectively, under the elastic Poisson transform $\bm{P}_e$. Notice that the space 
$\bm{h}_{-}^p(\B)$  is the Hardy space of Riesz fields. The space  $\bm{L}^2_-(\mathbb{S})$ of boundary values of $\bm{h}_{-}^2(\B)$  was characterized in \cite{KG} as a subspace of  $\bm{L}^2(\mathbb{S})$ with reproducing kernel.  By (\ref{LpDecom}) it follows that for $1< p<\infty$ 
$$\bm{h}_e^p(\B)=\bm{h}_{+}^p(\B)\oplus \bm{h}_{-}^p(\B)\oplus \bm{h}_{0}^p(\B).$$

The following is a restatement of Theorems \ref{espacio -1} and \ref{espacio 0}:
\begin{theorem}\label{Rieszfields and cross}
\begin{itemize}
 \item[1)]
Let $1<p<\infty$. Then $ \bm{u} \in  \bm{h}_{-}^p(\B)$  if and only if there exist $ g  \in W^{1,p}(\mathbb{S})$ such that $\bm{u}(x)=\nabla (Pg) (x)=P(\mathcal{L}_- g)(x)$. 
\item[2)] Let $1<p<\infty$. Then $\bm{u}\in \bm{h}_{0}^p(\B)$ if and only if there exists $g\in W^{1,p}(\S)$
 such that $\bm{u}(x)=x\times \nabla (Pg) (x) =P(\mathcal{L}_0g)(x)$. 
 \end{itemize}
 \end{theorem}
 \begin{corollary}
Let $1<p<\infty$. Then $ \bm{h}_{-}^p(\B)$ is isomorphic to $\bm{h}_{0}^p(\B).$
\end{corollary}

\begin{theorem}\label{harmonic in h+}
Let $1<p<\infty$,
 $\bm{u}\in \bm{h}_{+}^p(\B)$ with $\bm{u}$ harmonic if and only if $\bm{u}(x)=cx$ for some constant $c$.
 \end{theorem}
 
  \begin{proof} 
  Let $\bm{u}\in \bm{h}_{+}^p(\B)$ with $\bm{u}$ harmonic. Then there is $g \in W^{1,p}(\S)$ be such that $\bm{u}=P_e(\mathcal{L}_+g)=P(\mathcal{L}_+g)$. Now by Theorem \ref{espacio +1} we have that 
  $$\bm{u}(x)=[ 2 x \cdot \nabla (Pg)(x) + Pg(x)] x- |x|^2 \nabla (Pg)(x).$$
  It follows that 
  \begin{equation}\label{div-u}
  div\,\bm{u}(x)=3Pg(x)+ 5 x \cdot \nabla (Pg)(x)+2x\cdot \nabla(x\cdot \nabla (Pg)(x)).
  \end{equation}
  Let 
  $g=\sum_{\ell \geq 0} \sum_{|m| \leq \ell} a_{\ell , m} Y_{\ell, m}$ in $\dist$, so that  
  
$$ Pg(x)=\sum_{\ell \geq 0} \sum_{|m| \leq \ell} a_{\ell , m} Y_{\ell, m}(x),\;\; x \in \B.$$

It is easy to see from (\ref{div-u}) that
\begin{equation}\label{div-eq}
div\,\bm{u}(x)=\sum_{\ell \geq 0} \sum_{|m| \leq \ell} (2 \ell^2+5\ell+3)a_{\ell , m} Y_{\ell, m}(x).
\end{equation}

Now since $\Delta^*\bm{u}=\Delta\bm{u}= 0 $, we have that $\nabla div\,\bm{u}(x)=0 $, namely $div\,\bm{u}(x)$ is constant. From (\ref{div-eq})
we can conclude that $a_{\ell , m} =0$ for $\ell>0$ and then that $g=c$. This is that $\bm{u}(x)=cx$.
The sufficient condition is clear.
  \end{proof}

  \begin{corollary}
  Let $1<p<\infty$,
      $\bm{u}\in \bm{h}_{e}^p(\B)$  with $\bm{u}$ harmonic if and only if
      $$\bm{u}(x)=\nabla (Pg)(x)+x \times \nabla (P\widetilde{g})(x)+cx,$$
      for some $g, \widetilde{g} \in W^{1,p}(\mathbb{S}) $ and $c$ constant. Moreover, these are the unique solutions of the 
      Lam\'e system belonging to $\bm{h}_{e}^p(\B)$
for all $\lambda$ and $\mu$.
  \end{corollary}
  
   \begin{proof}
   It follows at once from Theorems \ref{Rieszfields and cross} and \ref{harmonic in h+}.
       
   \end{proof} 
  Now we will characterize the vector fields in $\bm{h}_{+}^p(\B)$. To motive the next Theorem recall from Lemma \ref{easyP_e} (and the notation there)
  that
  $$\bm{P}_e(\bm{E}_{\ell,m}^{+})(x)=(2\ell+1) \big{(}Y_{\ell,m}(x)x+\alpha_{\ell}(|x|^2-1)\nabla Y_{\ell,m}(x)\big{)}- \nabla Y_{\ell,m}(x).$$
  From Theorem \ref{descom orto}, this can be written as
  $$ \bm{P}_e(\bm{E}_{\ell,m}^{+})(x)=P(\bm{E}_{\ell,m}^{+})(x)+ 
  (|x|^2-1)\nabla P(\beta_{\ell}Y_{\ell,m})(x).$$
  where $\beta_{\ell}= (2\ell+1)\alpha_{\ell}+1.$ We are going to consider the multiplier $\mathcal{M}_{\beta_{\ell}}. $ 

 \begin{lemma}\label{Bounded}
 For $1<p<\infty$, the multiplier $\mathcal{M}_{\beta_{\ell}}: W^{1,p}(\S) \rightarrow  L^p(\S)$
  is bounded.
  \end{lemma}

\begin{proof}
    The proof follows directly from Theorem \ref{Strichartz}.
\end{proof}
  
\begin{theorem}\label{h+}
Let $1<p<\infty$. 
 If $\bm{u}\in \bm{h}_{+}^p(\B)$,  there exist $ g  \in W^{1,p}(\mathbb{S})$ such that 
 \begin{equation}\label{forma +}
     \bm{u}(x)=P (\mathcal{L}_+ g)(x)+(|x|^2-1)\nabla (P(\mathcal{M}_{\beta_{\ell}}g))(x) 
 \end{equation}
 Conversely, any function $\bm{u}$ as in \eqref{forma +} for $g\in W^{1,p}(\mathbb{S}),$ belongs to $\bm{h}_{+}^p(\B)$.
  \end{theorem} 

   \begin{proof}
   Let $\bm{u}\in \bm{h}_{+}^p(\B)$ by Theorem \ref{espacio +1} a)
   there is $g \in W^{1,p}(\S)$ be such that $\bm{u}(x)=\bm{P}_e(\mathcal{L}_+g)(x).$   
   Let $g=\sum_{\ell \geq 0} \sum_{|m| \leq \ell} a_{\ell , m} Y_{\ell, m}$ in $\dist$, 
   first note that by Lemma \ref{Bounded},  $\mathcal{M}_{\beta_{\ell}}:W^{1,p}(\mathbb{S})\rightarrow L^p(\mathbb{S})$
   is bounded. 
   Then
   \begin{align*}
     \bm{u}(x)&=\bm{P}_e(\mathcal{L}_+g)(x)=\sum_{\ell \geq 0} \sum_{|m| \leq \ell}a_{\ell , m}\bm{P}_e({E}_{\ell,m}^{+})(x)=\\ 
     &=\sum_{\ell \geq 0} \sum_{|m| \leq \ell}a_{\ell , m}P({E}_{\ell,m}^{+})(x)+(|x|^2-1)\sum_{\ell \geq 0} \sum_{|m| \leq \ell}a_{\ell , m}\nabla (P(\beta_{\ell}Y_{\ell,m}))(x)
     \\
       &=P(\mathcal{L}_+g)(x)+(|x|^2-1)\nabla (P(\mathcal{M}_{\beta_{\ell}}g))(x).
   \end{align*}
   The sufficient condition follows by reversing the steps above, to prove that $\bm{u}=\bm{P}_e(\mathcal{L}_+g)(x)$. 
   \end{proof}

 \begin{corollary}
  Let $1<p<\infty$,  $\bm{u}\in \bm{h}_{e}^p(\B)$, if and only if 
  $$\bm{u}(x)=\nabla (Pg)(x)+x \times \nabla (P\widetilde{g})(x)+ P(\mathcal{L}_+h)(x)+(|x|^2-1)\nabla (P(\mathcal{M}_{\beta_{\ell}}h))(x)$$
  for some $g, \widetilde{g}, h \in W^{1,p}(\mathbb{S}). $
 \end{corollary}  
\section{Appendix}
\subsection{Estimates for $P_e(x,\eta)$}

  \textit{Proof of Proposition \ref{estimaciones}.}


\textit{a):} Let $K$ be a compact in $\B$. There exists $\gamma >0$, depending only on $K$, such that
\begin{equation}\label{constan_compac}
   \gamma \leq  |tx - \eta| \leq 2, \hspace{0.3cm} t \in [0,1], \; \eta \in \S, \; x \in K. 
\end{equation}

From \eqref{ja_poisson_2} and \eqref{ja_poisson_4} we have

\begin{equation*}
4\pi \Phi(x,\eta)=\int_0^1\left(4 \pi P(tx,\eta)-1-3tx\cdot\eta\right)\frac{dt}{t^{1+\alpha}},
  \end{equation*}

  then 
  \begin{equation}\label{acotacion_poisson_e_3}
  \frac{\partial^2 \Phi}{\partial x_i\partial x_j}(x,\eta)=\int_0^1 \partial_i\partial_j P(tx,\eta)t^{1-\alpha}dt.
  \end{equation}

  Now,

  \begin{align}
   4\pi\partial_i P_i (x,\eta)  & = \frac{-2x_i}{\vert x-\eta\vert^3}-3\frac{(1-\vert x\vert^2)(x_i-\eta_i)}{\vert x-\eta\vert^5}, \nonumber \\
    4\pi\partial_j \partial_i P (x,\eta)  & =\frac{-2\delta_{ij}\vert x-\eta\vert^3+6x_i\vert x-\eta\vert(x_j-\eta_j)}{\vert x-\eta\vert^6} \nonumber \\
    & +3\frac{\vert x-\eta\vert^5\left((1-\vert x\vert^2)\delta_{ij}+2x_j(x_i-\eta_i)\right)}{\vert x-\eta\vert^{10}} \nonumber \\
    &+15 \frac{(1-\vert x\vert^2)(x_i-\eta_i)(x_j-\eta_j)\vert x-\eta\vert^3}{\vert x-\eta\vert^{10}}. \label{acotacion_poisson_e_2}
  \end{align}

  By using (\ref{ja_poisson_1})-(\ref{ja_poisson_3}), (\ref{constan_compac}) and (\ref{acotacion_poisson_e_2}) it is easy to see that $ \partial^\alpha_x\mathcal{P}_e(x,\eta)$ is uniformly bounded in $K\times\S$ for any compact
   $K\subset \B$.

\textit{b):}
By (\ref{acotacion_poisson_e_2}) and the fact that
\begin{equation}\label{desigualdad_1}
|tx-\eta| \geq 1-t|x| \geq 1-|x|, \hspace{0.3cm} t \in [0,1], \; \eta \in \S, \; x \in \B,
\end{equation}
we have
\begin{equation}\label{cotaD2Fi}
  \left|    \frac{\partial^2 \Phi}{\partial x_i\partial x_j}(x,\eta)  \right|  \lesssim \int_{0}^1 \frac{t^{1-\alpha}}{|tx-\eta|^4}dt.
\end{equation}
Then
by \eqref{desigualdad_1} and  \eqref{cotaD2Fi}, the kernel $\mathcal{L}$ of $L$ satisfies
\begin{equation}\label{cota L}
\vert \mathcal{L}_{i,j}(x,\eta)\vert\lesssim  (1-\vert x\vert^2)\int_0^1\frac{t^{1-\alpha}P(tx,\eta)}{\vert tx-\eta\vert^2}dt.
\end{equation}
Then $b)$ follows immediately.

  \textit{c)}: Let $\xi\in \S$ and $x\in\Gamma_\xi$. Then 
  \begin{equation}\label{maxL_1}
  \vert \bm{P}_e \bm{f} (x)\vert\leq \vert P \bm{f}(x)\vert+ \vert{\bm L} \bm{f}(x)\vert.
  \end{equation}
  By \eqref{ntmaximal vs maximal},
  \begin{equation}\label{maxarm}
  \vert P\bm{f}(x)    \vert\leq  P|\bm{f}|(x)  \leq  \mathcal{N} P|\bm{f}| (\xi)     \lesssim M[|\bm{f}|] (\xi).
  \end{equation}
 Since $tx\in \Gamma_\eta$ if $x\in \Gamma_\xi$ and $t\in [0,1]$,  then  by \emph{b)}, \eqref{maxarm} and  \eqref{integral} we have


  \begin{align}
      \vert{\bm L} \bm{f}(x)\vert & \lesssim \mathcal{N}P|\bm{f}|(\xi)(1-|x|^2) \int_0^1 \frac{t^{1-\alpha}}{(1-t^2|x|^2)^2} dt \nonumber \\
      & \lesssim M[|\bm{f}|] (\xi). \label{maxL}
  \end{align}

  Then the proof of \textit{(c)} follows from \eqref{maxL_1}, \eqref{maxarm} and \eqref{maxL}.
  \hfill $\square$

\subsection{Projections}

 \textit{Proof of Theorem \ref{desc Lp}.}
 Let $1<p<\infty$. The projections $\bm{\pi}_+$, $\bm{\pi}_-$ and $\bm{\pi}_0$ on $\bm{C}^\infty(\mathbb{S})$ are defined by  
  \begin{align*}
      \bm{\pi}_+ (\bm{f})(\xi) & = \sum_{\ell \geq 0} \sum_{|m| \leq \ell} \langle \bm{f} , \frac{\bm{E}^+_{\ell, m}}{\mu_\ell^+}  \rangle \frac{\bm{E}^+_{\ell, m}(\xi)}{\mu_\ell^+}, \\
      \bm{\pi}_- ( \bm{f})(\xi) & = \sum_{\ell \geq 1} \sum_{|m| \leq \ell} \langle \bm{f} , \frac{ \bm{E}^-_{\ell, m}}{\mu_\ell^-}  \rangle \frac{\bm{E}^-_{\ell, m}(\xi)}{\mu_\ell^-}, \\
      \bm{\pi}_0 (\bm{f})(\xi) & = \sum_{\ell \geq 1} \sum_{|m| \leq \ell} \langle \bm{f} ,\frac{ \bm{E}^0_{\ell, m}}{\mu_\ell^0}  \rangle \frac{\bm{E}^0_{\ell, m}(\xi)}{\mu_\ell^0}. \\
  \end{align*}


{\bf Boundedness of $\bm{\pi}_+$}:

 \vspace{.5 cm}
We have 

\begin{equation}\label{E_+_descompuesto}
 \bm{\pi}_+(\bm{f}) (\xi)= \sum_{i=1}^4 \bm{T}_i(\mathbf{f})(\xi),
\end{equation}
where 
\begin{align*}
  \bm{T}_1(\bm{f}) (\xi)  &= 
  \sum_{\ell \geq 0} \sum_{|m| \leq \ell} \frac{\ell +1}{2 \ell +1} \langle \bm{f} ,  \bm{Y}_{\ell, m}^{\vee} \rangle \bm{Y}_{\ell, m}^{\vee}(\xi),
 \\
 \bm{T}_2(\bm{f}) (\xi) &=- \sum_{\ell \geq 0} \sum_{|m| \leq \ell} \frac{1}{2 \ell +1}\langle \bm{f} ,  \nabla_\sigma  Y_{\ell ,m} \rangle \bm{Y}_{\ell ,m}^{\vee}(\xi), 
 \nonumber
 \\
 \bm{T}_3(\bm{f}) (\xi)&=- \sum_{\ell \geq 0} \sum_{|m| \leq \ell} \frac{1}{2 \ell +1} \langle \bm{f} , \bm{Y}_{\ell, m}^{\vee}  \rangle  \nabla_\sigma  Y_{\ell, m} (\xi),\\
 \bm{T}_4(\bm{f}) (\xi)&=\sum_{\ell \geq 0} \sum_{|m| \leq \ell} \frac{1}{(\ell +1)(2 \ell +1)} \langle \bm{f} ,  \nabla_\sigma  Y_{\ell, m}  \rangle  \nabla_\sigma  Y_{\ell ,m} (\xi).
\end{align*}

 Notice that for $\bm{f}\in \bm{C}^\infty(\mathbb{S})$, we can write
\begin{equation*}
\bm{T}_1(\bm{f})=\mathcal{M}_{\frac{\ell +1}{2 \ell +1}}(\bm{f}\cdot \bm{\nu})^{\vee}.
\end{equation*}
Since by Theorem \ref{Strichartz} the operator  $\mathcal{M}_{\frac{\ell +1}{2 \ell +1}}$ is continuous in $L^p(\S)$, it follows immediately that $\bm{T}_1$ has a bounded extension in $\bm{L}^p-$norm.

Next,  for $\bm{f}\in \bm{C}^\infty(\mathbb{S})$, we can write

$$-\bm{T}_3(\bm{f})(\xi)= 
\sum_{\ell=0}^\infty \sum_{m=-\ell}^\ell \frac{1}{2 \ell +1}  \widehat{(\bm{f}\cdot\bm{\nu})}_{\ell,m}   \mathbf{\nabla}_\sigma  Y_{\ell, m} (\xi), 
  $$

Since $\sum_{\ell=0}^\infty \sum_{m=-\ell}^\ell  \widehat{(\bm{f}\cdot\bm{\nu})}_{\l,m}    Y_{\ell, m} (\xi)$ converges in the Sobolev space $W^{1,p}(\S)$, we have that

$$-\bm{T}_3(\bm{f})(\xi)= \nabla_\sigma \mathcal{M}_{\frac{1}{2 \ell +1}}(\bm{f}\cdot\bm{\nu}).$$
Then, using Theorems \ref{Strichartz} and \ref{theorem_bakry}, 

\begin{equation*}
\Vert\bm{T}_3(\bm{f})\Vert_{\bm{L}^p}   \lesssim\Vert (-\Delta)^{1/2}(\bm{f}\cdot\bm{\nu})\Vert_{L^p}=\Vert\mathcal{M}_{\alpha_\ell}(\bm{f}\cdot\bm{\nu})\Vert_{L^p}\lesssim \Vert \bm{f}\Vert_{\bm{L}^p},
\end{equation*}

where $$\alpha_\ell=\frac{\ell^{1/2} \left( \ell+1\right)^{1/2}}{2\ell+1}.$$
 It also follows that $\bm{T}_2$ is $\bm{L}^p-$bounded since its adjoint is  $\bm{T}_3$, which is $\bm{L}^q-$ bounded for any $1<q<\infty$.

Finally, to study $\bm{T}_4$, we observe that for $\bm{f} \in \bm{C}^\infty (\S)$,

$$ \bm{T}_4 (\bm{f}) =   \nabla_\sigma \left( \sum_{\ell \geq 1} \sum_{|m| \leq \ell}  \frac{1}{(\ell +1)(2 \ell +1)} \langle \bm{f} ,  \nabla_\sigma  Y_{\ell,  m}  \rangle   Y_{\ell ,m}  \right).$$
By  Theorem \ref{theorem_bakry},
\begin{align}\label{pasoprevio}
\nonumber & \int_{\mathbb{S}} \left|  \bm{T}_4 (\bm{f}) (\xi)   \right|^p d \sigma (\xi)  \\ \nonumber & \lesssim \int_{\mathbb{S}} \left| \left(-\Delta \right)^{1/2}    \left( \sum_{\ell \geq 1} \sum_{|m| \leq \ell} \frac{1}{(\ell +1)(2 \ell +1)} \langle \bm{f} ,  \nabla_\sigma  Y_{\ell k}  \rangle   Y_{\ell k} (\cdot) \right) (\xi)  \right|^p d \sigma (\xi)  \\\nonumber
 & = \int_{\mathbb{S}} \left|  \sum_{\ell \geq 1} \sum_{|m| \leq \ell}  \frac{\ell^{1/2} \left( \ell + 1   \right)^{1/2}}{(\ell +1)(2 \ell +1)} \langle \bm{f} ,  \nabla_\sigma  Y_{\ell ,m}  \rangle   Y_{\ell ,m} (\xi)   \right|^p d \sigma (\xi) \\
 &=\int_{\mathbb{S}} \left|  \sum_{\ell \geq 1} \sum_{|m| \leq \ell}  \frac{\ell^{1/2} \left( \ell + 1   \right)^{1/2}}{(\ell +1)(2 \ell +1)} \langle \bm{f} ,  \nabla_\sigma  Y_{\ell ,m}  \rangle \bm{Y}_{\ell ,m}^{\vee}(\xi)\right|^p d \sigma (\xi).
\end{align}

Let  $\bm{\widetilde{T}}_2$ be   the operator defined by
\begin{equation*}\label{T_4_2}
\bm{\widetilde{T}}_2( \xi)=\sum_{\ell \geq 1} \sum_{|m| \leq \ell} \frac{\ell^{1/2} \left( \ell + 1  \right)^{1/2}}{(\ell +1)(2 \ell +1)} \langle \mathbf{f} ,  \nabla_\sigma  Y_{\ell ,m}  \rangle \bm{Y}_{\ell ,m}^{\vee}(\xi),
\end{equation*}
and its adjoint
\begin{equation*}\label{T_4_2}
\bm{\widetilde{T}}_3( \xi)=\sum_{\ell \geq 1} \sum_{|m| \leq \ell} \frac{\ell^{1/2} \left( \ell + 1   \right)^{1/2}}{(\ell +1)(2 \ell +1)} \langle \mathbf{f} , \bm{Y}_{\ell ,m}^{\vee} \rangle \nabla_\sigma  Y_{\ell ,m}(\xi) .
\end{equation*}

As done for $\bm{T}_3$, it can proved that $\bm{\widetilde{T}}_3$ is $\bm{L}^p-$bounded for any $1<p<\infty$, and hence the same holds for  $\bm{\widetilde{T}}_2$.
Thus from \eqref{pasoprevio},
\begin{equation*}
\Vert\bm{T}_4 (\bm{f})\Vert_{\bm{L}^p} \lesssim\Vert\bm{f}\Vert_{\bm{L}^p}.
\end{equation*}
This completes the proof that $ \bm{\pi}_+$ is bounded in $\bm{L}^p(\S)$. The proof of the $L^p$- continuity of $ \bm{\pi}_-$ is completely analogous, and since $ \bm{\pi}_0=Id- \bm{\pi}_+- \bm{\pi}_-$ 
on $ \bm{C}^\infty (\S)$, the proof of the theorem is complete.

 \hfill $\square$

Acknowledgement. We thank Fabricio Macía for helpful conversations and for sharing his notes on vector spherical harmonics.

\vskip 0.08in
\noindent --------------------------------------
\vskip 0.10in
\begin{minipage}[t]{7.7cm}
\noindent {\tt Juan Antonio Barcel\'o}

\noindent Departamento de Matem\'atica e Inform\'atica 

\noindent aplicadas a las Ingenier{\'i}as Civil y Naval

\noindent Universidad Polit\'ecnica de Madrid

\noindent Madrid, 28040, Spain

\noindent {\tt e-mail}: {\it juanantonio.barcelo@upm.es}

\vskip 0.20in

\noindent {\tt Emilio Marmolejo-Olea}

\noindent Instituto de Matem\'aticas Unidad Cuernavaca 

\noindent Universidad Nacional Aut\'onoma de M\'exico

\noindent A.P. 273-3 Admon. 3,

\noindent Cuernavaca, Morelos, 62251 M\'exico

\noindent {\tt e-mail}: {\it emiliomo\@@im.unam.mx}

\vskip 0.20in

\end{minipage}
\hfill
\begin{minipage}[t]{7.7cm}

\noindent {\tt Salvador Per\'ez-Esteva}

\noindent  Instituto de Matem\'aticas Unidad Cuernavaca 

\noindent Universidad Nacional Aut\'onoma de M\'exico

\noindent A.P. 273-3 Admon. 3,

\noindent  Cuernavaca, Morelos, 62251 M\'exico

\noindent {\tt e-mail}: {\it spesteva@im.unam.mx }

\vskip 0.20in

\noindent {\tt Mary Cruz Vilela}

\noindent Departamento de Matem\'atica e Inform\'atica 

\noindent aplicadas a las Ingenier{\'i}as Civil y Naval

\noindent Universidad Polit\'ecnica de Madrid

\noindent Madrid, 28040, Spain

\noindent {\tt e-mail}: {\it maricruz.vilela@upm.es}
\end{minipage}
\end{document}